\documentclass{amsart}

\usepackage[margin=1.51in]{geometry}
\usepackage{amsmath}
\usepackage{amsthm}
\usepackage{amssymb}
\usepackage{amsfonts}
\usepackage{amsrefs}
\usepackage{color}
\usepackage{tikz}
\usepackage{dynkin-diagrams}

\tikzstyle{shaded}=[fill=red!10!blue!20!gray!30!white]
\tikzstyle{shaded line}=[double=red!10!blue!20!gray!30!white, double distance=1.5mm, draw=black]
\tikzstyle{unshaded}=[fill=white]
\tikzstyle{unshaded line}=[double=white, double distance=1.5mm, draw=black]
\tikzstyle{Tbox}=[circle, draw, thick, fill=white, opaque,]
\tikzstyle{empty box}=[circle, draw, thick, fill=white, opaque, inner sep=2mm]
\tikzstyle{background rectangle}= [fill=red!10!blue!20!gray!40!white,rounded corners=2mm] 
\tikzstyle{on}=[very thick, red!50!blue!50!black]
\tikzstyle{off}=[gray]

\tikzstyle{traces}=[scale=.2, inner sep=1mm]
\tikzstyle{quadratic}=[scale=.4, inner sep=1mm, baseline]
\tikzstyle{annular}=[scale=.7, inner sep=1mm, baseline]
\tikzstyle{make triple edge size}= [scale=.4, inner sep=1mm,baseline] 
\tikzstyle{icosahedron network}=[scale=.3, inner sep=1mm, baseline]
\tikzstyle{ATLsix}=[scale=.25, baseline]
\tikzstyle{TL12}=[scale=.15,baseline]
\tikzstyle{PAdefn}=[scale=.7,baseline]
\tikzstyle{TLEG}=[scale=.5,baseline]

\makeatletter

\newtheorem{lemma}{Lemma}[section]

\newtheorem{definition*}{Definition}
\newtheorem{theorem}[lemma]{Theorem}

\newtheorem{question}[lemma]{Question}

\makeatother

 \title{The Boolean intervals of Chevalley type are strongly non group-complemented}
\author{Sebastien Palcoux and Pablo Spiga}
\address{Yau Mathematical Sciences Center, Tsinghua University, Beijing, China}
\email{sebastien.palcoux@gmail.com}
\address{Dipartimento di Matematica Pura e Applicata, University of Milano-Bicocca, Milano, Italy}
\email{pablo.spiga@unimib.it} 

\begin{document}

\maketitle

\begin{abstract}
Let $G$ be a finite Chevalley group and $B$ a Borel subgroup. Then the interval $[B,G]$ in $\mathcal{L}(G)$ is Boolean. We prove, using Zsigmondy's theorem, that for any element $P$ in the open interval $(B,G)$, its lattice-complement $P^{\complement}$ is not a group-complement.
\end{abstract}

\section{Preliminaries} \label{pre}
A bounded lattice $(L, \vee, \wedge)$ is called \textit{complemented} if any element $a \in L$ admits a \textit{complement} $b \in L$ (i.e. $a \vee b = \hat{1}$ and $a \wedge b = \hat{0}$).  An interval of groups $[B,G]$ in $\mathcal{L}(G)$ will be called \textit{lattice-complemented} if it is complemented as a lattice. It will be called \textit{group-complemented} if it is lattice-complemented and if for any $P \in [B,G]$ there is a lattice-complement $Q$ which is a \textit{group-complement}, i.e. $PQ = QP$ (so that $PQ = G$).  The interval $[B,G]$ is called \textit{strongly non group-complemented} if for any $P$ in the open interval $(B,G)$, there is no (lattice-complement which is a) group-complement.

We refer to \cite{car} for the notions of Dynkin diagram, Chevalley group, Borel subgroup, BN-pair, parabolic subgroup, unipotent subgroup.  We will use the classification of the finite (connected) Dynkin diagrams and the order of the corresponding Chevalley groups over the finite field $\mathbb{F}_q$ with $q$ a prime power. Let us recall them:  
$$A_n \dynkin[mark=o, edge length=0.6cm]{A}{} \hspace{0.7cm} F_4 \dynkin[mark=o, edge length=0.6cm]{F}{4} \hspace{0.7cm} G_2 \dynkin[mark=o, edge length=0.6cm]{G}{2}$$
$$B_n \dynkin[mark=o, edge length=0.6cm]{B}{} \hspace{1.435cm} E_6 \dynkin[mark=o, edge length=0.6cm]{E}{6}$$
$$C_n \dynkin[mark=o, edge length=0.6cm]{C}{} \hspace{0.855cm} E_7 \dynkin[mark=o, edge length=0.6cm]{E}{7}$$  
$$D_n \dynkin[mark=o, edge length=0.6cm]{D}{} \hspace{0.514cm} E_8 \dynkin[mark=o, edge length=0.6cm]{E}{8}$$

\footnotesize  

$$\begin{array}{|l|l|l|l|l|}
\hline
G & A_n(q), n \ge 1 & B_n(q), n>1 & C_n(q), n>2 & D_n(q), n>3   \\
  \hline
|G| & \frac{q^{\frac{1}{2} n(n + 1)}}{(n + 1, q - 1)} \prod_I\left(q^i - 1\right)  & \frac{q^{n^2}}{(2, q - 1)} \prod_I\left(q^i - 1\right) & \frac{q^{n^2}}{(2, q - 1)} \prod_I\left(q^i - 1\right)  & \frac{q^{n(n-1)}}{(4, q^n-1)} \prod_I\left(q^i - 1\right) \\
  \hline 
I & \{2,3, \dots , n+1\} & \{2,4, \dots , 2n\} & \{2,4, \dots , 2n\} & \{n\} \cup \{2,4, \dots , 2n-2\} \\
\hline
\end{array} $$

$$\begin{array}{|l|l|l|l|l|l|}
\hline
G & E_6(q) & E_7(q)& E_8(q) & F_4(q) & G_2(q)  \\
  \hline  
|G| & \frac{q^{36}}{(3,q-1)}\prod_I(q^i-1) & \frac{q^{63}}{(2,q-1)}\prod_I(q^i-1) & q^{120}\prod_I(q^i-1) & q^{24} \prod_I(q^i-1) & q^{6}\prod_I(q^i-1)  \\
  \hline 
I & \{2,5,6,8,9,12\} & \{2,6,8,10,12,14,18\} & \{2,8,12,14,18,20,24,30\} & \{2,6,8,12\} &  \{2,6\} \\
\hline
 \end{array}$$
\normalsize

\noindent Let $G$ be a group with a BN-pair of rank $n$. Then the interval $[B,G]$ is Boolean of rank $n$ \cite[Theorems 8.3.2 and 8.3.4]{car}. A Chevalley group $G$ admits a BN-pair \cite[Proposition 8.2.1]{car} of rank the rank of $G$, i.e. the number $n$ of vertices in its Dynkin diagram (denoted) $X_n$. Let $B$ be a Borel subgroup. Then let call $[B,G]$ a \textit{Boolean interval of Chevalley type}.    

Let $V$ be the set of vertices of $X_n$. Let $W$ be a subset of $V$ and let $X_W$ be the sub-diagram of $X_n$ whose vertices are $W$ and whose edges are given by $X_n$ (so that $X_V = X_n$).  Note that $X_W$ decomposes into connected components which are connected Dynkin diagrams.

\begin{theorem} \label{levi} The lattice-isomorphism between the power-set $\mathcal{P}(V)$ and the interval $[B,G]$ realizes as a map $W \mapsto P_W$ such that $P_W = U_W \rtimes L_W$ (Levi Decomposition) where $U_W$ (called the unipotent subgroup) decomposes into a product of \emph{root subgroups} (each of them being isomorphic to the additive group of the field $\mathbb{K}$) where each element admits a unique decomposition, and $L_W$ (called a Levi subgroup) is isomorphic to the direct product of the Chevalley groups (over $\mathbb{K}$) corresponding to the connected components of $X_W$. 
\end{theorem}
\begin{proof}
For the Levi Decomposition, see \cite[Section 8.5]{car}. About the unipotent subgroup, see \cite[pages 68 and 119, and Theorem 5.3.3 (ii)]{car}. Finally, the last assertion about $L_W$ follows from its definition \cite[page 119]{car}.
\end{proof}

\section{Result}

\begin{lemma} \label{subb}
The connected sub-diagrams of the Dynkin diagrams are classified as follows:
\footnotesize    
$$\begin{array}{|l|l|l|l|l|}
\hline
\text{Dynkin diagram} & A_n \ (n \ge 1) & B_n \ (n>1) & C_n \ (n>2) & D_n \ (n>3)   \\
  \hline
\text{Connected} & A_m \ (m \le n) & A_m \ (m < n), & A_m \ (m < n), B_2,  & A_m \ (m < n),  \\

\text{sub-diagrams}           &              &  B_m \ (1 < m \le n) & C_m \ (2 < m \le n)               & D_m \ (3 < m \le n) \\
\hline                            
\end{array} $$

$$\begin{array}{|l|l|l|l|}
\hline
\text{Dynkin diagram} & E_n \ (6 \le n \le 8) & F_4 & G_2   \\
  \hline
\text{Connected} & A_m \ (m < n), D_{n-1}, & A_1,A_2, B_2,  & A_1,  \\

\text{sub-diagrams}                &  E_m \ (6 \le m \le n)  & B_3,C_3, F_4  & G_2  \\
\hline                            
\end{array} $$
\normalsize
\end{lemma}
\begin{proof} A checking watching the Dynkin diagrams in Section \ref{pre}.
\end{proof}

Let $X_n$ be a (connected) Dynkin diagram and let $I(X_n)$ denote the set $I$ associated to $X_n(q)$ in the tables of Section \ref{pre}.  

\begin{lemma} \label{sub}
Let $Y_m$ be a connected sub-diagram of $X_n$ with $m<n$. Then $$\max(I(Y_m))<\max(I(X_n)).$$  
\end{lemma}
\begin{proof} A checking using all the previous tables.
\end{proof}

\begin{theorem} \label{zig} Let $a>1$ and $r>2$ be integers. Then there is a prime number $p$ (called \emph{primitive prime divisor}) that divides $a^r - 1$ and does not divide $a^k - 1$ for any positive integer $k < r$, except for $a=2$ and $r=6$ where $2^6-1 = (2^2-1)^2(2^3-1)$.
\end{theorem}
\begin{proof} It is a particular case of Zsigmondy's theorem \cite{zi}.
\end{proof}

\begin{theorem} \label{main}
Let $G$ be a finite Chevalley group over the finite field $\mathbb{F}_q$ (with $q$ a prime power) and let $B$ be a Borel subgroup. Let $P$ be an element in the open interval $(B,G)$. Then its lattice-complement $P^{\complement}$ is not a group-complement.
\end{theorem} 
\begin{proof} Let $X_n$ be the Dynkin diagram of $G$, and let $V$ denote the vertices of $X_n$. By Theorem \ref{levi}, there is $W \in \mathcal{P}(V) \setminus \{\emptyset, V\}$ such that $P = P_W$ and $P^{\complement} = P_{W^{\complement}}$. 

If $P^{\complement}$ is a group-complement, i.e. $PP^{\complement}=P^{\complement}P = G$, then by Product Formula: $$ |G| \cdot |B| = |P| \cdot |P^{\complement}| =  |L_W| \cdot |L_{W^{\complement}}| \cdot |U_W| \cdot |U_{W^{\complement}}|,$$
with $U_W$ and $L_W$ as for Section \ref{pre}.  Take $r = \max(I(X_n))$ and let $p$ be a primitive prime divisor of $q^r-1$ (coming from Theorem \ref{zig}). First $p$ cannot divide $|U_W| \cdot |U_{W^{\complement}}|$ because (by Theorem \ref{levi}) this last number is a power of the order $q$ of $\mathbb{F}_q$. So $p$ must divide $|L_W| \cdot |L_{W^{\complement}}|$ which (by Theorem \ref{levi}) is a product of order of Chevalley groups whose Dynkin diagrams are proper sub-diagrams of $X_n$. The contradiction follows from Theorem \ref{zig} and Lemma \ref{sub} (so that $P^{\complement}$ is not a group-complement), except when $q=2$ and $r=6$. 

In this last case, $X_n$ must be $A_5$, $B_3$, $C_3$, $D_4$ or $G_2$. We will explain how to rule out $A_5$ in a \emph{generic way} (the same argument works for the others). We know that $$2^{-15}|A_5(2)| = (2^2-1)(2^3-1)(2^4-1)(2^5-1)(2^6-1).$$ Now, $2^5-1$ has a primitive prime divisor. So we have $\{ L_W , L_{W^{\complement}} \} = \{A_1(2),A_4(2)\}$. But $$q^{-11}|A_1(2)| \cdot |A_4(2)| = (2^2-1)^2(2^3-1)(2^4-1)(2^5-1) = (2^4-1)(2^5-1)(2^6-1).$$  
It follows that $(2^2-1)(2^3-1)$ divides $\frac{|G| \cdot |B|}{|P| \cdot |P^{\complement}|}=1$, contradiction. 
\end{proof}

%(i.e. the subgroup of invertible upper triangular matrices). For any two subspaces $V,W \subseteq \mathbb{K}^n$ of same dimension there exists an element $g \in G$ transforming $V$ into $W$. Now,

Note that there is a shorter and more \emph{illuminating} proof working for the classical groups. Let us sketch it for $G = GL(\mathbb{K}^n)$, it works as well in general. Let $B$ be a Borel subgroup of $G$. For any $P$ in the open interval $(B,G)$, $P$ is a parabolic subgroup and there is a subspace $V \subset \mathbb{K}^n$ such that $P=G_V$ (the stabilizer subgroup). But, for every subspace $W \subset \mathbb{K}^n$ with $\dim(W) = \dim(V)$, there is $g \in G$ such that $g(V) = W$. \emph{Assume} that $PP^{\complement} = P^{\complement}P= G$, then $g=ab$ with $a \in P^{\complement}$ and $b \in P$. But $W=g(V) = ab(V) = a(V)$ because $b \in P=G_V$. Conclusion, for every subspace $W \subset \mathbb{K}^n$ with $\dim(W) = \dim(V)$, there is $a \in P^{\complement}$ such that $a(V) = W$, so $P^{\complement} = G$, \emph{contradiction}.

Unfortunately the above argument does not work for the exceptional groups, that is why the current paper provides a \emph{uniform} case-by-case proof using Zsigmondy's theorem, working for every Chevalley group. For being even more general, let us ask the following: 

\begin{question} Is every Boolean interval $[H,G]$ of $\mathcal{L}(G)$, for every finite simple group $G$, strongly non group-complemented, up to finitely many exceptions?
\end{question}
\noindent To answer this question, we need to classify the finite simple groups $G$ having a group factorization $G = AB$ such that $[A \cap B, G]$ is Boolean; which reduces to a maximal factorization with no extra intermediate, in the sense that the open interval $(A \cap B, G)$ is equal to $\{A,B\}$. We obtained by GAP \cite{gap} the following list for $(G,A,B)$ when $|G| < 2 \cdot 10^6$: 
\begin{itemize}
\item $(A_6, \ A_5, \ A_5)$,
\item $(A_8, \ A_7, \ 2^3:A_1(7))$,  
\item $(M_{12}, \ M_{11}, \ M_{11})$, 
\item $(C_2(2^2), \ A_1(2^4):2, \ A_1(2^4):2)$,  
\item $ (C_3(2), \ A_8:2, \ ^2A_2(3^2):2)$.  
\end{itemize}
A full classification should be accessible by using \cite{LPS}.

%Can we extend Theorem \ref{main} to any finite simple group $G$ with $B$ coming from a BN-pair (when it exists)?
%\begin{corollary} Consider $P_1, P_2 \in [B,G]$ with $|P_1| \le |P_2|$. Then $$P_1P_2=P_2P_1 \Leftrightarrow P_1 \le P_2.$$ \end{corollary}

\section{Acknowledgments} 
Thanks to the organizers of the conference AGTA 2019 (Lecce, Italy) where the authors met. The first author thanks the second for answering a question coming from a lack of understanding of a paragraph in a referee report (written in 2016 for a preprint of \cite{bp}). This short note is a redaction of this answer. 

\end{document}